\documentclass[12 pt,twoside]{amsart}

\usepackage{amsmath,amsfonts,amsthm,amssymb}
\usepackage{bbm}

\newtheorem{theorem}{Theorem}[section]
\newtheorem{lemma}[theorem]{Lemma}

\newtheorem{corollary}[theorem]{Corollary}

\theoremstyle{definition}

\theoremstyle{remark}

\numberwithin{equation}{section}

\begin{document}

\title[Linear semigroups with coarsely dense orbits]{Linear semigroups with coarsely dense orbits}

\author[H. Abels]{H. Abels}
\address{Fakult\"{a}t f\"{u}r Mathematik, Universit\"{a}t Bielefeld, Postfach 100131, D-33501 Bielefeld, Germany}
\email{abels@math.uni-bielefeld.de}

\author[A. Manoussos]{A. Manoussos}
\address{Fakult\"{a}t f\"{u}r Mathematik, SFB 701, Universit\"{a}t Bielefeld, Postfach 100131, D-33501 Bielefeld, Germany}
\email{amanouss@math.uni-bielefeld.de}
\urladdr{http://www.math.uni-bielefeld.de/~amanouss}
\thanks{During this research the second author was fully supported by SFB 701 ``Spektrale Strukturen und
Topologische Methoden in der Mathematik" at the University of Bielefeld, Germany.}

\date{}



\begin{abstract}
Let $S$ be a finitely generated abelian semigroup of invertible linear operators on a finite dimensional real or complex vector space $V$. We show that every coarsely
dense orbit of $S$ is actually dense in $V$. More generally, if the orbit contains a coarsely dense subset of some open cone $C$ in $V$ then the closure of the orbit
contains the closure of $C$. In the complex case the orbit is then actually dense in $V$. For the real case we give precise information about the possible cases for the
closure of the orbit.
\end{abstract}

\maketitle

\section{Introduction}
Recall that a subset $Y$ of a metric space $(X,d)$ is called \textit{D-dense} in $X$ if the open balls of radius $D$ with centers at points of $Y$ cover $X$. The subset
$Y$ is called \textit{coarsely dense} in $X$ if there is a positive number $D$ such that $Y$ is $D$-dense in $X$.

\begin{theorem} \label{main}
Let $V$ be a finite dimensional real or complex vector space endowed with a norm. Let $S$ be a subsemigroup of $GL(V)$ generated by a finite set of commuting elements.
Then every coarsely dense orbit of $S$ in $V$ is actually dense.
\end{theorem}

We show a more general result, theorem \ref{general}, as follows. Suppose an orbit $\mathcal{O}$ of $S$ has a subset which is coarsely dense in some open cone $C$ in
$V$. Then we show that the closure of $\mathcal{O}$ is a cone which contains $C$. We also give detailed information about the structure of the orbits of the closure
$\overline{S}$ of $S$. The topology we refer to without further comment here and throughout the whole paper will be the Euclidean topology. We will sometimes also use
the Zariski topology. But then we will be more specific and use expressions like Zariski-dense etc. Under our more general hypotheses we prove furthermore the following
claim, see theorem \ref{th41}. The closure $\overline{S}$ of $S$ in $GL(V)$ is actually a group. There are only finitely many maximal $\overline{S}$-invariant vector
subspaces of $V$. Let $U$ be the complement of their union. Then $U$ has a finite number of connected components which are open cones in $V$. For every $v\in U$ the
orbit $\overline{S}\cdot v$ is a union of connected components of $U$, in particular an open cone.

The original motivation for work on this paper came from the theory of linear operators. N.S. Feldman \cite{Fe2} initiated the study of linear operators $T:X\to X$ on a
Banach space $X$ which have a coarsely dense orbit. Such orbits arise naturally when perturbing vectors with dense orbits. For instance, if the orbit of a vector $x\in
X$ is dense in $X$ and $y\in X$ is a vector with a bounded orbit then the orbit $\mathcal{O}(x+y,T)$ is coarsely dense. In the same paper, Feldman showed that it is
possible for an orbit to be coarsely dense without being dense in $X$. He also showed that if there exists a coarsely dense orbit then there exists (a possibly
different) vector with a dense orbit; see \cite[Theorem 2.1]{Fe2}. This result is based on the fact that in infinite dimensional Banach spaces there exists a sequence of
norm-one vectors $\{ x_n\}$ and a positive number $\varepsilon$ such that $\| x_n - x_k \| \geq \varepsilon$ for all $n\neq k$. Coarse density of orbits of a single
operator is a phenomenon that occurs only in infinite dimensional vector spaces but coarse density of orbits of finitely generated semigroups may occur also on finite
dimensional vector spaces. In contrast with the case of a single operator on an infinite dimensional space, every coarsely dense orbit of a subsemigroup of $GL(V)$
generated by a finite set of commuting elements is dense as the main result of the present paper shows.

The plan of the paper is the following. In section 2 we start by showing that under a quite natural and obviously necessary condition the closure $\overline{S}$ of a
subsemigroup $S$ of a vector group is actually a group (Theorem \ref{posneg}). This easily implies the same criterion for subsemigroups of compactly generated abelian
locally compact topological groups, by Pontryagin's structure theorem, corollary \ref{cor24}. It has as corollaries a criterion for density, corollaries \ref{SdenseR}
and \ref{cor25}(3), and for cocompactness, corollary \ref{cor25}(1). These statements  may be regarded as duality statements for subsemigroups of such abelian groups.
The general philosophy is: Test by applying continuous homomorphisms to $\mathbb{R}$. Thus the separating hyperplane theorem for closed convex cones is used in an
essential way. These results of section 2 may be of independent interest.

In section 3 we show our main theorem in the more general form mentioned above, namely for cones with inner points. At a crucial point (lemma \ref{lemma37}) we use a
general result about algebraic actions of algebraic groups. In section 4 we give detailed information about the orbits of $\overline{S}$, part of which was described
above. Section 5 contains remarks and questions.

\section{A density criterion for semigroups}
\begin{theorem} \label{posneg}
Let $S$ be a subsemigroup of $\mathbb{R}^n$ with the property that whenever a linear form $l:\mathbb{R}^n\to\mathbb{R}$ has a positive value on $S$ it also has a
negative value on $S$. Then the closure of $S$ is a subgroup of $\mathbb{R}^n$.
\end{theorem}

All closed subgroups $A$ of $\mathbb{R}^n$ are of the following form. There is a decomposition $\mathbb{R}^n=V_1\oplus V_2\oplus V_3$ into three (not necessarily
non-zero) vector subspaces $V_1,V_2,V_3$ of $\mathbb{R}^n$ such that $A=V_1\oplus \Gamma_2$, where $\Gamma_2$ is a lattice in $V_2$. Thus

\begin{corollary} \label{SdenseR}
Let $S$ be a subsemigroup of $\mathbb{R}^n$ such that $l(S)$ has dense image in $\mathbb{R}$ for every non-zero linear form $l:\mathbb{R}^n\to\mathbb{R}$. Then $S$ is
dense in $\mathbb{R}^n$.
\end{corollary}

We have good criteria for when a subsemigroup of $\mathbb{R}$ is dense, see corollary \ref{denscycl} and lemma \ref{polS}. For a more general form of this corollary see
corollary \ref{cor25}(3).

\medskip

\noindent\textit{Proof of theorem \ref{posneg}.} We may assume that $S$ spans $\mathbb{R}^n$, equivalently that $l(S)$ is non-zero for every non-zero linear form
$l:\mathbb{R}^n\to\mathbb{R}$. Otherwise we consider the vector subspace of $\mathbb{R}^n$  which is the intersection of the kernels of all linear forms
$l:\mathbb{R}^n\to\mathbb{R}$ which vanish on $S$. Let $P(S)=\{ r\cdot s\,;\, r\geq 0,\,s\in S\}$ be the positive hull of $S$. We claim that the closure
$\overline{P(S)}$ of $P(S)$ is not only a closed cone but also convex. Namely, since $S$ is a semigroup, if $s_1$, $s_2$ are elements of $S$ all positive rational linear
combinations of $s_1$ and $s_2$ are in $P(S)$ and hence $\overline{P(S)}$ contains the convex cone generated by $s_1$ and $s_2$. This easily implies that
$\overline{P(S)}$ is convex.

We next claim that $\overline{P(S)}=\mathbb{R}^n$. Otherwise there would be a non-zero linear form $l:\mathbb{R}^n\to\mathbb{R}$ such that $l(\overline{P(S)})\geq 0$, by
the separating hyperplane theorem. But this contradicts our hypotheses about $S$. Let us fix a norm $\| \cdot \|$ on $\mathbb{R}^n$. Then the set of directions $\{
\frac{s}{\| s\|}\,;\, s\in S,\, s\neq 0\}$ of $S$ is dense in the norm $1$ sphere $S^1=\{ v\in\mathbb{R}^n\,;\, \| v\| =1\}$ since $P(S)$ is dense in $\mathbb{R}^n$. Let
$x\neq 0$ be an element of $S$. We claim that $-x\in \overline{S}$. This implies that $-\overline{S}\subset \overline{S}$ and hence the theorem. Given $x$, there is a
basis $e_1,\ldots,e_n$ of $\mathbb{R}^n$ consisting of elements of $S$ such that $x=\alpha_1 e_1+\ldots +\alpha_n e_n$ with $\alpha_i <0$ for $i=1,\ldots,n$, by the
density of the set of directions of $S$, since if $x$ is a linear combination of a basis such that all the coefficients have negative sign then the same holds for a
nearby basis and hence also for one with nearby directions. Let $\Gamma$ be the lattice in $\mathbb{R}^n$ generated by $e_1,\ldots,e_n$. There is a sequence of natural
numbers $m_j\to\infty$ such that $-m_jx \mod \Gamma$ converges to the identity element in the compact group $\mathbb{R}^n/ \Gamma$. So for every $\varepsilon >0$ there
is a large natural number $m$ and integers $b_1,\ldots,b_n$ such that $\| -mx-b_1e_1-\ldots -b_ne_n \| <\varepsilon$. Comparing coefficients we see that $b_i>0$ for
$i=1,\ldots, n$. So the element $b_1e_1+\ldots +b_ne_n +(m-1)x$ of $S$ has distance less than $\varepsilon$ from $-x$, which implies our claim. \qed

\medskip

Theorem \ref{posneg} has as special case the following corollary.

\begin{corollary} \label{denscycl}
Let $S$ be a subsemigroup of $\mathbb{R}$ which contains both a positive and a negative real number. Then $S$ is either dense in $\mathbb{R}$ or a cyclic subgroup of
$\mathbb{R}$.
\end{corollary}

We shall use the theorem in the following form.

\begin{corollary} \label{cor24}
Let $G$ be a compactly generated abelian locally compact topological group, e.g. an abelian Lie group whose group of connected components is finitely generated. Let $S$
be a subsemigroup of $G$. Suppose that every continuous homomorphism $f:G\to\mathbb{R}$ which has a positive value on $S$ also has a negative value on $S$. Then the
closure of $S$ is a subgroup of $G$.
\end{corollary}
\begin{proof}
The group $G$ has a unique maximal compact subgroup $K$ and $G/K$ is isomorphic to a direct sum of a vector group $\mathbb{R}^m$ and a lattice $\mathbb{Z}^l$, by
Pontryagin's structure theorem. Let $\pi:G\to G/K$ be the natural projection. We can think of $\mathbb{R}^m\oplus\mathbb{Z}^l$ as a subgroup of $\mathbb{R}^n$ with
$n=m+l$. The closure of $\pi (S)$ is a subgroup of $G/K$ by theorem \ref{posneg}. The corollary then follows from the following facts. The mapping $\pi$ is proper so the
image $\pi (\overline{S})$ of the closure $\overline{S}$ of $S$ is closed, hence a subgroup of $G/K$. And $\overline{S}\cap K$ is a closed subsemigroup of the compact
group $K$, hence a subgroup.
\end{proof}

A subsemigroup $S$ of an abelian topological group $G$ is said to be \textit{cocompact} if there is a compact subset $K$ of $G$ such that $G=S\cdot K$. The
subsemigroup $S$ of $G$ is called \textit{properly discontinuous} if it has no accumulation point in $G$. And $S$ is called \textit{crystallographic} if it both
cocompact and properly discontinuous.

\begin{corollary} \label{cor25}
Let $G$ be a compactly generated abelian locally compact topological groups and let $S$ be a subsemigroup of $G$.
\begin{enumerate}
\item Suppose $f(S)$ is cocompact in $\mathbb{R}$ for every non-zero continuous homomorphism $f:G\to\mathbb{R}$. Then $S$ and $\overline{S}$ are cocompact in $G$ and
$\overline{S}$ is a subgroup of $G$.

\item If $S$ is furthermore properly discontinuous then $S$ is a closed discrete crystallographic subgroup of $G$.

\item If $f(S)$ is dense in $\mathbb{R}$ for every non-zero continuous homomorphism $f:G\to\mathbb{R}$ then $G/K$ is a vector group isomorphic to $\mathbb{R}^m$ for some
$m$, $\overline{S}$ is a subgroup of $G$ and $\overline{S}\cdot K=G$, where $K$ is the maximal compact subgroup of $G$.
\end{enumerate}
\end{corollary}
\begin{proof}
(1) For every non-zero continuous homomorphism $f:G\to\mathbb{R}$ the subsemigroup $f(S)$ of $\mathbb{R}$ is cocompact, hence has both a positive and a negative value.
So $\overline{S}$ is a subgroup of $G$. As in the proof of corollary \ref{cor24} we may assume that $G=\mathbb{R}^m\oplus\mathbb{Z}^l\subset\mathbb{R}^n$, with $n=m+l$,
by computing modulo the maximal compact subgroup $K$ of $G$. A subgroup of $G$ is then cocompact if and only if it spans $\mathbb{R}^n$ as a vector space. This shows
that $\overline{S}$ is cocompact in $G$ and hence also $S$, since if $G=\overline{S}\cdot L$ for some compact subset $L$ of $G$ then $G=S\cdot L_1$ for some compact
neighborhood $L_1$ of $L$.

(2) If $S$ is furthermore properly discontinuous then $S$ is closed in $G$, hence $S=\overline{S}$ is a closed crystallographic, in particular discrete, subgroup of $G$.

(3) Finally, if $f(S)$ is dense in $\mathbb{R}$ for every non-zero continuous homomorphism $f:G\to\mathbb{R}$ then $\overline{S}$ is a subgroup of $G$ by (1). Also, the
factor $\mathbb{Z}^l$ in Pontryagin's structure theorem $G/K \cong \mathbb{R}^m\oplus\mathbb{Z}^l$ does not occur, so $G/K \cong \mathbb{R}^m$, and $\overline{S}\cdot
K=G$ by corollary \ref{SdenseR}.
\end{proof}

\section{Proof of the main theorem}
We shall prove the following theorem which contains theorem \ref{main} as the special case $C=V$. Note that we always refer to the Euclidean topology, unless we
explicitly mention the Zariski topology.

\begin{theorem} \label{general}
Let $V$ be a finite dimensional real or complex vector space. Let $S$ be a subsemigroup of $GL(V)$ generated by a finite set of commuting elements. Suppose there is
an open cone $C$ in $V$ and an orbit $\mathcal{O}=S\cdot v_0$ of $S$ such that $\mathcal{O}\cap C$ is coarsely dense in $C$. Then
\begin{enumerate}
\item The closure $\overline{S}$ of $S$ in $GL(V)$ is an open subgroup of the Zariski closure of $S$. In particular, $\overline{S}$ has a finite number of connected
components since every real algebraic group has only a finite number of connected components, see \cite[V.24.6(c)(i)]{Bo}.

\item The map $\overline{S}\to V$, $g\mapsto gv_0$, is an analytic diffeomorphism of $\overline{S}$ onto an open cone in $V$. In particular
$\dim \overline{S}=\dim_{\mathbb{R}} V$.

\item The closure of the orbit $\mathcal{O}$ is a cone in $V$ which contains $C$.
\end{enumerate}
\end{theorem}

We will give more information on the cones $\overline{S}\cdot v_0$ and $\overline{S\cdot v_0}$ below; see theorem \ref{th41}.

The proof of theorem \ref{general}  is given in steps, proceeding from special cases to more general cases. We start with

\begin{lemma} \label{polS}
Let $S$ be a finitely generated subsemigroup of $\mathbb{R}$. Then either $S$ is dense in $\mathbb{R}$ or every bounded interval in $\mathbb{R}$ contains only a finite
number of elements of $S$. In the latter case, the number of elements of $S\cap [-n,n]$ is bounded by a polynomial in $n$.
\end{lemma}
\begin{proof}
If $S$ contains both a positive and a negative number then the lemma follows from corollary \ref{denscycl}. So suppose $S$ is finitely generated and contained in
$[0,\infty)$. Let $s_0$ be a minimal positive generator of $S$. If $S$ is generated by $t$ elements, there are at most $n^t$ elements of $S$ in $[0,ns_0)$.
\end{proof}

This implies the case $\dim_{\mathbb{R}} V=1$ of theorem \ref{general}.

\begin{corollary} \label{addense}
Let $\mathbb{R}^*_{+}$ be the multiplicative group of positive real numbers. Endow it with the metric induced from the Euclidean metric of $\mathbb{R}$. Then every
coarsely dense finitely generated subsemigroup of $\mathbb{R}^*_{+}$ is dense.
\end{corollary}
\begin{proof}
Let $S$ be a finitely generated coarsely dense subsemigroup of $\mathbb{R}^*_{+}$. Then the interval $[1,x]$, $x>1$, contains at least $c\cdot x$ elements of $S$ for
some positive constant $c$. Thus the interval $[0,\log x]$ contains at least $c\cdot x$  elements of the finitely generated subsemigroup $\log S$ of the additive
group $\mathbb{R}$, hence is dense in $\mathbb{R}$, by the preceding lemma.
\end{proof}

Let $\mathbb{R}^n_{+}=\{ (x_1,\ldots,x_n)\, ;\, x_i>0\}$. Then $\mathbb{R}^n_{+}=(\mathbb{R}^*_{+})^n$ is a group under componentwise multiplication. By a
\textit{cone} in $\mathbb{R}^n_{+}$ we mean a cone in $\mathbb{R}^n$ that is contained in $\mathbb{R}^n_{+}$. In geometric terms, a subset $C$ of $\mathbb{R}^n_{+}$
is a cone if and only if for every non-zero $x\in C$ the open ray $\mathbb{R}^*_{+}\cdot x$ is contained in $C$. We endow $\mathbb{R}^n_{+}$ with the metric induced
from the Euclidean metric on $\mathbb{R}^n$.

\begin{lemma} \label{SdenseRn}
Let $S$ be a finitely generated subsemigroup of $\mathbb{R}^n_{+}$ with the property that for some $a\in \mathbb{R}^n_{+}$ the orbit $S\cdot a=\{ s\cdot a\, ;\, s\in
S\}$ contains a coarsely dense subset of some open cone in $\mathbb{R}^n_{+}$. Then $S$ is dense in $\mathbb{R}^n_{+}$.
\end{lemma}
\begin{proof}
We will use corollary \ref{SdenseR} in the following form. If $\alpha (S)$ has dense image in $\mathbb{R}^*_{+}$ for every non-trivial continuous homomorphism $\alpha
: \mathbb{R}^n_{+}\to \mathbb{R}^*_{+}$ then $S$ is dense in $\mathbb{R}^n_{+}$. This follows from corollary \ref{SdenseR} by passing to exponentials since continuous
homomorphisms $\alpha : \mathbb{R}^n_{+}\to \mathbb{R}^*_{+}$ correspond bijectively to linear maps $l: \mathbb{R}^n\to \mathbb{R}$ under $l\mapsto \exp \circ l\circ
\log$. We will show that $\alpha (S)$ has dense image in $\mathbb{R}^*_{+}$ for every  $\alpha : \mathbb{R}^n_{+}\to \mathbb{R}^*_{+}$ with $\alpha
(x_1,\ldots,x_n)=x_1^{\alpha_1} \ldots x_n^{\alpha_n}$ for $(\alpha_1,\ldots,\alpha_n)\neq 0\in\mathbb{R}^n$. Let $C$ be an open cone in $\mathbb{R}^n_{+}$ which
contains a coarsely dense subset of the orbit $S\cdot a$. There are two cases to consider:

(1) $\sum \alpha_i \neq 0$. We may assume that $\sum \alpha_i >0$, by passing to $\alpha^{-1}$, if necessary, and that $\sum \alpha_i <1$, by passing to a positive
multiple $r\cdot (\alpha_1,\ldots,\alpha_n)$, $r>0$, of $(\alpha_1,\ldots,\alpha_n)$, which changes the image of $\alpha$ by the automorphism $x\mapsto x^r$ of
$\mathbb{R}^*_{+}$. Then $\alpha$ is positively homogeneous of degree $\sum \alpha_i$ and $\frac{\partial \alpha}{\partial x_i}=\alpha_i \cdot x_i^{-1}\cdot \alpha$
is positively homogeneous of degree $\sum \alpha_i -1 <0$ for $i=1,\ldots,n$. Hence $\alpha$ grows monotonously to $+\infty$ on every ray $R:=\mathbb{R}^*_{+} \cdot
x$, $x\in \mathbb{R}^n_{+}$, and there is an open cone $C'$ in $\mathbb{R}^n_{+}$ containing our ray $R$ such that $\alpha$ is distance decreasing on the set of
points of $C'$ of sufficiently large Euclidean norm. If $C'$ is contained in the cone $C$ it follows that $\alpha (S\cdot a)=\alpha (S)\cdot \alpha(a)$ is coarsely
dense in $\mathbb{R}^*_{+}$, hence also $\alpha (S)$. Thus $\alpha (S)$ is dense in $\mathbb{R}^*_{+}$, by corollary \ref{addense}. The other case is

(2) $\sum \alpha_i =0$. Let $S^1=\{ x\in\mathbb{R}^n \, ;\, \| x\| =1\}$ be the norm $1$ sphere in $\mathbb{R}^n$ and $S^1_{+}=S^1\cap \mathbb{R}^n_{+}$. The function
$\alpha$ is constant on every ray in $\mathbb{R}^n_{+}$. By our hypothesis about coarse density the set of points of $S^1_{+}$ for which the corresponding ray
$\mathbb{R}^*_{+} \cdot x$ intersects $S\cdot a$ is dense in $S^1_{+} \cap C$. It follows that the closure of $\alpha (S\cdot a)=\alpha (S)\cdot \alpha(a)$ contains
the open set $\alpha (C)$. Hence $\alpha(S)$ is dense in $\mathbb{R}^*_{+}$ by applying lemma \ref{polS} to the finitely generated subsemigroup $\log \alpha (S)$ of
$\mathbb{R}$ or to $-\log \alpha (S)$.
\end{proof}

The next case we consider is that $S$ is diagonalizable over the complex numbers. Then $V$ decomposes into a direct sum of $S$-invariant one- or two-dimensional real
vector subspaces $V_1,\ldots,V_r,V_{r+1},\ldots,V_d$ with the following properties. Each $V_i$, $i\leq r$, is one-dimensional and $S$ acts by multiplication by
scalars. Each $V_i$ for $i>r$ is two-dimensional and can be endowed with the structure of a one-dimensional complex vector space such that $S$ acts by multiplication
by complex scalars. We thus have an embedding of $S$ into the following group $G$. Let $G$ be the group of diagonal $d\times d$-matrices whose first $r$ diagonal
entries are real and non-zero and whose last $d-r$ entries are complex non-zero. Let $\delta_i :G\to\mathbb{R}^*$, resp. $\mathbb{C}^*$, be the projection to $i$-th
diagonal component. So $S$ acts in the following way on $V=V_1 \oplus\ldots\oplus V_d$; $s(v_1,\ldots,v_d)=(\delta_1 (s)v_1,\ldots,\delta_d (s) v_d )$ for $s\in S$.

\begin{lemma} \label{Sdiag}
Suppose $S$ is diagonalizable over the complex numbers. Suppose furthermore that there is an orbit of $S$ which contains a co\-a\-rse\-ly dense subset of some open
cone. Then the closure of $S$ is a subgroup of $G$ and contains the connected component $G^0$ of the identity of $G$.
\end{lemma}
\begin{proof}
Let $v_0$ be a point of $V$ and let $C$ be an open cone in $V$ such that $S\cdot v_0 \cap C$ is coarsely dense in $C$. Let $V^{\neq 0}$ be the set of vectors in
$V=V_1 \oplus\ldots\oplus V_{d}$ all of whose components are non-zero. The group $G$ acts simply transitively on $V^{\neq 0}$. Clearly, $v_0\in V^{\neq 0}$. We have
an absolute value map $\text{abs} :V\to [0,\infty )^{d}$, $v=(v_1,\ldots,v_{d})\mapsto (|v_1|,\ldots,|v_{d}|)$ if we choose isomorphisms $V_i\simeq\mathbb{R}$ for
$i\leq r$ and $V_i\simeq\mathbb{C}$ for $i>r$. We also have an absolute value map on $G$, $|\cdot |:G\to (\mathbb{R}^*_{+})^n$. The pair $(|\cdot |,\text{abs})$ is
equivariant, that is $\text{abs}(g\cdot v)=|g| \text{abs}(v)$ for $g\in G$ and $v\in V$. The image of $C$ under the map $\text{abs}$ is an open cone in
$(\mathbb{R}^*_{+})^n$. The image of $S\cdot v_0$ under the map $\text{abs}$ is the orbit of $\text{abs} (v_0)$ under $|S|$ and contains a coarsely dense subset of
$\text{abs}(C)$, since $\text{abs}$ is distance non-expanding. Note that the notion of coarse density is independent of the norm on $V$ we choose, since any two norms
on $V$ are Lipschitz equivalent. It follows from lemma \ref{SdenseRn} that $|S|$ is dense in $(\mathbb{R}^*_{+})^d$. This implies that the closure $\overline{S}$ of
$S$ is a group, since the kernel $K$ of the homomorphism $|\cdot |:G\to (\mathbb{R}^*_{+})^d$ is compact and $\overline{S}\cap K$ is a closed subsemigroup of $K$ and
hence a subgroup.

We now claim that the closure $\overline{S}$ of $S$ contains an open subset of $G$ and hence is an open subgroup of $G$ which finishes the proof. We will actually show
that $\overline{S}\cdot v_0$ is dense in $C$ which implies our claim, since $G\to V^{\neq 0}$, $g\mapsto gv_0$, is a diffeomorphism. The group $G$ is the direct product
of its subgroup $(\mathbb{R}^*_{+})^d$ and its maximal compact subgroup $K$. We denote the two components of an element $g\in G$ by $|g|$ and $\arg (g)$, respectively,
since this decomposition is just the polar decomposition of the diagonal entries in $\mathbb{C}^*$ and $\mathbb{R}^*$, respectively. The image of $\overline{S}$ under
$|\cdot |$ is closed, since  $|\cdot |$ is a proper map, hence $|\overline{S}|=(\mathbb{R}^*_{+})^d$. So for every positive real number $r$ there is an element $g\in
\overline{S}$ such that $|g|$ is the homothety $M_r$ with $M_r(x)=r \cdot x$. Taking an appropriate power of $g$ we obtain elements $g^n$ with $\arg (g^n)$ arbitrarily
close to the identity. It follows that for every point $x\in C$ there are arbitrarily large powers $g^n$ of $g$ such that $g^{-n}x\in C$. Then there is a point $y\in
S\cdot v_0$ at distance at most $D$ from $g^{-n}x$, where $D$ is the constant in the definition of coarse density. Thus the point $g^ny\in \overline{S}\cdot v_0$ is at
distance $r^{n}\cdot D$ from $x$, since $\| g^ny-x \|=\| g^n(y-g^{-n}x)\|= r^n \| \arg (g^n) (y-g^{-n}x)\|\leq r^{n}\cdot D$. This shows that $\overline{S}\cdot v_0$ is
dense in $C$ if we take $r<1$.
\end{proof}

We now come to the proof of the general case of theorem \ref{general}.

\begin{lemma} \label{lemma36}
Under the hypotheses of theorem \ref{general} the orbit $\mathcal{O}:=S\cdot v_0$ is Zariski-dense in $V$.
\end{lemma}
\begin{proof}
The claim is that the zero polynomial is the only polynomial function on $V$ which vanishes on $\mathcal{O}$. Thus let $f$ be a non-zero polynomial function which
vanishes on $\mathcal{O}$. Let $f_n$ be the homogeneous component of $f$ of highest degree. Then there is an open subset $U$ of $C$ such that $f_n$ has no zero on
$\overline{U}$. We may assume that $\overline{U}$ is compact. Then $| f(tx)|=|f_n(tx)+(f(tx)-f_n(tx))|\geq t^{n-1} (t\cdot \min \{ f_n(x)\,;\, x\in \overline{U}\} - \max
\{ |f(x)-f_n(x)|\,;\, x\in \overline{U}\})$ for $x\in \overline{U}$ and $t>0$, so $f(tx)$ has no zero for $t\gg 0$ and  $x\in \overline{U}$, contradicting the fact that
$\mathcal{O} \cap C$ is coarsely dense in $C$.
\end{proof}

\begin{lemma} \label{lemma37}
Let $G$ be the Zariski closure of $S$ in $GL(V)$. Then $G\to V$, $g\mapsto gv_0$, is a diffeomorphism of $G$ onto an open subset of $V$.
\end{lemma}
\begin{proof}
By a theorem on algebraic actions \cite[$\S$ 3.18]{BoTi} the map $G\to V$, $g\mapsto gv_0$, to the orbit $Gv_0$ is a submersion and the orbit $Gv_0$ is an open subset
- with respect to the Euclidean topology - in its Zariski closure. Furthermore the isotropy group $G_v=\{ g\in G\,;\, gv=v\}$ is independent of the chosen point $v$
of the orbit, since $G_{gv}=gG_vg^{-1}$ and $G$ is abelian, and hence must be trivial, since the identity is the only element of $GL(V)$ that fixes every point of an
open set.
\end{proof}

\begin{proof}[Proof of theorem \ref{general}]
Let $G$ be the Zariski closure of $S$ in $GL(V)$. Then $G$ is an abelian real linear algebraic group and hence is the direct product of its subgroup $G_s$ of semisimple
elements and its subgroup $G_u$ of unipotent elements. Let $\pi_s :G\to G_s$ and $\pi_u :G\to G_u$ be the corresponding projection homomorphisms. Let $V_u$ be the
$G$-invariant vector subspace $\sum_{u\in G_u} (1-u)V$ of $V$. The quotient space $V_s:=V/V_u$ has a natural representation of $G$ whose kernel is $G_u$ and the
representation of $G$ on $V_s$ is semisimple. Let $p_s:V\to V_s$ be the natural projection. Then the point $p_s(v_0)$ has the orbit $S\cdot p_s(v_0)=\pi_s(S)p_s(v_0)$.
This orbit intersects the open cone $p_s(C)$ in a coarsely dense subset of $p_s(C)$. It follows from lemma \ref{Sdiag} that the closure of $\pi_s(S)$ is a subgroup of
$G_s$ and $\dim \overline{\pi_s(S)}=\dim_{\mathbb{R}}V_s$. On the other hand $\dim G_s=\dim_{\mathbb{R}}V_s$ by the preceding lemma, applied to the action of $G_s$ on
$V_s$. So $G_s^0\subset \overline{\pi_s(S)}\subset G_s$, where $H^0$ denotes the connected component of the identity in a topological group $H$. The group $G_s$ was
described in the paragraph before lemma \ref{Sdiag}. In particular  $G_s^0$ contains the group $\mathbb{R}^*_{+} \cdot \mathbbm{1}$ of homotheties $M_r$ with
$M_r(x)=r\cdot x$ for $r\in \mathbb{R}^*_{+}$.

Next we will show that the closure $\overline{S}$ of $S$ in $G$ is a subgroup of $G$ and  $\overline{S}\cdot K=G$ where $K$ is the maximal compact subgroup of $G$. It
suffices to show that for every non-zero continuous homomorphism $f:G\to\mathbb{R}$ the image $f(S)$ is dense in $\mathbb{R}$, by corollary \ref{cor25}(3). Note that
this corollary also implies that $G/K$ is a vector group, in particular connected, which can also be seen from the facts that $G_u$ is connected and $G_s$ is -
considered as a Lie group - a direct product of groups $\mathbb{R}^*$ and $\mathbb{C}^*$. Thus let $f:G\to\mathbb{R}$ be a non-zero continuous homomorphism. Let
$\varphi:Gv_0\to \mathbb{R}$ be the analytic map defined by $\varphi (gv_0)=f(g)$. Note that
\[
\varphi (gx)=f(g)+\varphi (x)
\]
for $g\in G$ and $x\in Gv_0$. Recall that $Gv_0$ is open in $V$. We know that $\mathbb{R}^*_{+} \cdot \mathbbm{1}\subset G^0$. There are two cases to consider, namely
that the restriction $f|\mathbb{R}^*_{+} \cdot \mathbbm{1}$ of $f$ to $\mathbb{R}^*_{+} \cdot \mathbbm{1}$ is the zero map or not.

\medskip

\noindent\textit{Case 1.} Suppose $f|\mathbb{R}^*_{+} \cdot \mathbbm{1}$ is the zero map. Then $\varphi$ is constant on every ray $\mathbb{R}^*_{+}x$ in $Gv_0$. Let
$C$ be our open cone such that $S\cdot v_0\cap C$ is coarsely dense in $C$. Then for every open subset $U$ of $C$ the set of $x\in U$ for which the corresponding ray
$\mathbb{R}^*_{+}x$ intersects $S\cdot v_0$ is dense in $U$, so $f(S)=\varphi (S\cdot v_0)$ contains a dense subset of the open set $\varphi (U)$. Note that $\varphi$
is an open map, namely the composition of a diffeomorphism $Gv_0\to G$ and a non-zero homomorphism  $G\to\mathbb{R}$. It follows from lemma \ref{polS} that $f(S)$ is
dense in $\mathbb{R}$.

\medskip

\noindent\textit{Case 2.} The other case is that $f|\mathbb{R}^*_{+} \cdot \mathbbm{1}$ is non-zero. We may assume that $f (e^t \mathbbm{1})=t$ for every
$t\in\mathbb{R}$. Let $x\in C$ and let $B(x,\varepsilon)$ be the ball of radius $\varepsilon$ with center $x$. Then $B(e^t x,e^t \varepsilon)=M_{e^t} B(x,\varepsilon)$
contains a point of  $S\cdot v_0$ if $e^t \varepsilon \geq D$, where $D$ is the constant in the definition of coarse density. There is a constant $c$ such that $|\varphi
(y)-\varphi (z)|\leq c\| y-z\|$ for every $y,z\in B(x,\varepsilon)$ since $\varphi$ is continuously differentiable. It follows that for every $t\in\mathbb{R}$
\begin{align*}
\begin{split}
\varphi(B(e^t x,e^t \varepsilon)) &=\varphi(M_{e^t} B(x,\varepsilon))=f(M_{e^t})+\varphi(B(x,\varepsilon))\\
                                         &=t+\varphi(B(x,\varepsilon))
\end{split}
\end{align*}
is contained in the interval in $\mathbb{R}$ with center $t$ and radius $c\cdot\varepsilon$. So for sufficiently large $t\in\mathbb{R}$ there is an element of
$\varphi(S\cdot v_0)=f(S)$ of distance at most $D\cdot c\cdot e^{-t}$ from $t$. It follows from lemma \ref{polS} that $f(S)$ is dense in $\mathbb{R}$.

The group $G$ is the direct product of $G_u$ and $G_s$ and $G_s$ is the direct product of its polar part, a product of $\mathbb{R}^*_{+}$, and its maximal compact
subgroup, a product of groups $\{ \pm 1\}$ and $S^1$. The same proof as at the end of lemma \ref{Sdiag} shows that $\overline{S}$ contains the connected component
$G^0$ of $G$.

Now $\mathbb{R}^*_{+} \cdot \mathbbm{1}$ is contained in $\overline{S}$, hence the orbit of every point under $\overline{S}$ is a cone, i.e.
$\mathbb{R}^*_{+}$-invariant. It follows that  $\overline{S}\cdot v_0$ contains a dense subset of $C$, so the closure of $\overline{S}\cdot v_0$ contains $C$. Note that
it can happen that $C$ is not contained in the orbit $\overline{S}\cdot v_0$. E.g. if $C=\mathbb{R}$, $v_0\neq 0$ then $\overline{S}\cdot v_0$ does not contain 0. Recall
that $\overline{S}$ is the closure of $S$ in the group $GL(V)$, hence contained in $\mathbb{R}^*$ for $V=\mathbb{R}$.  The general picture can be seen from theorem 4.1.
\end{proof}

\section{Orbit structure}
We determine the structure of the open orbits of $\overline{S}$. The following theorem will be applied for the group $G^*=\overline{S}$ of theorem \ref{general}. Parts
of these results are in \cite{AM1}.

\begin{theorem} \label{th41}
Let $G^*$ be an abelian subgroup of $GL(V)$ which has an orbit with an inner point. Let $G$ be the Zariski closure of $G^*$. Then the following claims hold:
\begin{enumerate}
\item The group $G^*$ is an open subgroup of $G$ and contains the connected component $G^0$ of $G$.

\item There is only a finite set of maximal $G^0$-invariant vector subspaces of $V$. They are also $G$-invariant. These are of real codimension $1$ or $2$ in $V$. Let
$H_1,\ldots,H_r$ be those of codimension $1$ and $H_{r+1},\ldots,H_d$ be those of codimension $2$. We have
\[
r+2(d-r)\leq \dim_{\mathbb{R}} V.
\]

\item Let $U$ be the complement of $\bigcup_{i=1}^d H_i$ in $V$ and let $v$ be a point of $U$. Then the map $G\to U$, $g\mapsto gv$, is an analytic diffeomorphism.

\item The quotient group $G/G^0$ is isomorphic to $(\mathbb{Z}/2\mathbb{Z})^r$. For $v\in U$ the orbit $G^0v$ is the connected component of $v$ in $U$. The closure of $G^0v$ is
the intersection of $r$ half spaces. Namely, for each $i=1,\ldots,r$ there is exactly one open  half space defined by $H_i$ which contains $v$, say $C_i$. Then
$\overline{G^0v}=\bigcap_{i=1}^r \overline{C_i}$.

\item If $V$ has a structure as complex vector space such that every $h\in G^*$ is complex linear, then $r=0$, $G$ is connected and hence $G^0=G^*=G$. Then $U$ is the
only open $G^*$-orbit and thus it is dense in $V$.
\end{enumerate}
\end{theorem}
\begin{proof}
The group $G^*$ contains an open subset of $G$, by lemma \ref{lemma37}. Thus, being a group, $G^*$ is open in $G$ and hence a closed subgroup of $G$. This implies
(1). The group $G$ is an abelian real linear algebraic group so it has a finite number of connected components. It is also a direct product of its subgroup $G_s$ of
semisimple elements and its subgroup $G_u$ of unipotent elements.

For the remaining claims we start with special cases.
\medskip

\noindent (a) $G_s=\mathbb{R}^* \cdot \mathbbm{1}$. There is a non-zero linear form $l:V\to\mathbb{R}$ which is fixed by $G_u$. So $G_u$ leaves every hyperplane
parallel to the kernel $H$ of $l$ invariant. There is a vector $v$ with a somewhere dense orbit, so $l(v)\neq 0$. Now $\varphi_v:G\to Gv$, $g\mapsto gv$, is a
diffeomorphism onto an open subset of $V$ which maps  $G_s=\mathbb{R}^* \cdot \mathbbm{1}$ onto the line $\mathbb{R}v$ with the origin removed and $G_u$ to the affine
hyperplane $H+v$. Comparing dimensions we see that $G_uv$ is a neighborhood of $v$ in $H+v$ and hence $G_uv$ is an open subset of $H+v$. On the other hand, the orbit
of any vector under a unipotent Zariski closed subgroup of $GL(V)$ is closed \cite[Lemma 5.3]{Moore}. It follows that $G_uv=H+v$. Now $G_u(tv)=tG_uv=tv+H$, so the
orbit of any vector $v\notin H$ is $V\setminus H$. This shows that $H$ is the only $G$ or $G^0$-invariant hyperplane in $V$. This is the case $r=d=1$ of our theorem.
\medskip

\noindent (b) $V$ has a complex structure such that $G^*$ consists of complex linear maps and $G_s=\mathbb{C}^* \cdot \mathbbm{1}$. Thus $G$ is connected and hence
$G^0=G^*=G$. Exactly the same arguments as above show that $Gv=V\setminus H$ for every $v\notin H$. In particular $H$ is the only maximal $G$-invariant (real or
complex) vector subspace of $V$. This is the case $r=0$, $d=1$, of our theorem.
\medskip

\noindent (c) The general case. Every isotypic component of the $G_s$-module $V$ is of the type described in (a) or (b), since the hypothesis that the orbit of $G^*$ has
an inner point carries over to the isotypic components and then to their semisimple quotients $V$ modulo $\sum_{g\in G_u} (1-g)V$, so the other irreducible algebraic
subgroups of $\mathbb{C}^*$ do not occur in the simple quotients of $V$. Note that if a simple quotient is $\mathbb{C}$ with the action of $\mathbb{C}^* \cdot
\mathbbm{1}$, considered as a real vector space, then the isotypic component has a (unique up to complex conjugation) structure as a complex vector space and every $g\in
G$ acts by complex linear automorphisms, since $G$ is commutative. So, back to the general case, $V$ has a decomposition into isotypic modules $V_1\oplus \ldots \oplus
V_r\oplus V_{r+1}\oplus \ldots \oplus V_d$ where $G_s$ acts on $V_i$ as $\mathbb{R}^* \cdot \mathbbm{1}$ for $i\leq r$ and for $i>r$ the $G_s$-module $V_i$ has a
structure as complex vector space, $G_s$ acts as $\mathbb{C}^* \cdot \mathbbm{1}$ and $G$ acts by complex linear automorphisms. Let $H$ be a maximal $G$-invariant real
vector subspace of $V$. Then there is an index $i$, $1\leq i\leq d$, such that $H=\bigoplus_{j\neq i} V_j\oplus H_i^*$ and $H_i^*$ is a maximal $G$-invariant subspace of
$V_i$. We thus are with $H_i^*\subset V_i$ in case (a) or (b). In particular, for a given index $i$ there is exactly one such subspace. Let us call it $H_i$. The natural
representation of $G$ on $V / \oplus H_i^*= \bigoplus V_i/H_i^*$ has image in $\bigoplus_{i=1}^r \left( \mathbb{R}^* \cdot \mathbbm{1}_{V_i/H_i^*} \right)\oplus
\bigoplus_{i=r+1}^d \left( \mathbb{C}^* \cdot \mathbbm{1}_{V_i/H_i^*}\right)$ and has an orbit with an inner point. So the dimension of $G_s$ is $r+2(d-r)=\dim V/ \oplus
H_i^*$. It follows that $\oplus H_i^*$ has the same codimension in $V$ as $G_u$ has in $G$. But $G_u$ leaves $\oplus H_i^*$ invariant. Now let $v_0$ be a point of $V$
whose $G^*$-orbit has an inner point. Then $v_0\notin \bigcup H_i$, in other words, for every component $v_0^i$ of $v_0$ in $V_i$ we have $v_0^i\notin H_i^*$. The map
$G\to Gv_0$, $g\mapsto gv_0$, is a diffeomorphism of $G$ onto an open subset of $V$, by lemma \ref{lemma37}. So, comparing dimensions, we see that the orbit of $G_u$ in
$v_0\oplus \bigoplus H_i^*$ is open. But $G_uv_0$ is also closed since every orbit of a unipotent algebraic group is closed, so $G_uv_0=v_0\oplus \bigoplus H_i^*$. On
the other hand $G=G_s\times G_u$, and $G_s$ is an algebraic subgroup of $\bigoplus_{i=1}^r \left(\mathbb{R}^* \cdot \mathbbm{1}_{V_i} \right)\oplus \bigoplus_{i=r+1}^d
\left( \mathbb{C}^* \cdot \mathbbm{1}_{V_i}\right)$ of the same dimension, hence they are equal. This shows claim (3).

If $V$ has a complex structure, then every simple quotient of the $G$-module $V$ again has an open orbit, so we are in case (b), which implies (5). Part (4) spells
out the connected components of $G$, namely $G/ G^0 \cong G_s/G_s^0 \cong (\mathbb{R}^*)^r / (\mathbb{R}^*_{+})^r \cong (\mathbb{Z}/2\mathbb{Z})^r$.
\end{proof}

\section{Open problems and remarks}
\noindent\textit{Question 1.}\,\, Does theorem \ref{posneg} also hold for subsemigroups of topological vector spaces? Or at least corollary \ref{SdenseR}?

\medskip

G. Soifer asked us the following question. Let $S$ be a finitely generated commutative linear semigroup. Suppose $S$ has an orbit which is coarsely dense along a
spanning set of rays. Here a set $A$ is called coarsely dense along a subset $Y$ of a metric space $(X,d)$ if the open balls of radius $D$ with centers at points of $A$
cover $Y$ for some positive number $D$. Does this imply that there is an orbit which is dense in some open set? The answer is no to this more general question. Here are
two examples.

1) Consider the one-parameter group $G$ of diagonal positive matrices in $SL_2(\mathbb{R})$. Every dense finitely generated subsemigroup of $G$ is coarsely dense along
the two positive axes, but has no somewhere dense orbit.

2) Consider the group $G$ of upper triangular $2\times 2$-matrices whose diagonal entries are equal and positive. Let $S$ be a finitely generated subsemigroup of $G$
whose closure are all the matrices of $G$ whose off diagonal entries are integers. This semigroup has an orbit which is (coarsely) dense along infinitely many rays, but
has no somewhere dense orbit.

The following question is open.

\medskip

\noindent\textit{Question 2.}\,\,  Can one find a finite number $R_1,\ldots,R_m$ of rays such that every finitely generated commutative linear semigroup $S$  which has a
coarsely dense orbit along $\bigcup_{i=1}^mR_i$ also has a somewhere dense orbit?

\medskip

\noindent \textbf{Acknowledgement.} We thank the referee for valuable suggestions which helped us to improve the presentation.

\end{document}